\documentclass[12pt,oneside,reqno]{amsart}

\usepackage{amsmath,amssymb,amsthm,textcomp}
\usepackage{amsfonts,graphicx}
\usepackage[mathscr]{eucal}
\usepackage{color}

\theoremstyle{definition}

\numberwithin{equation}{section}

\newcommand{\ncom}{\newcommand}

\ncom{\beq}{\begin{equation}}
\ncom{\eeq}{\end{equation}}
\ncom{\bea}{\begin{eqnarray*}}
\ncom{\eea}{\end{eqnarray*}}
\ncom{\beqa}{\begin{eqnarray}}
\ncom{\eeqa}{\end{eqnarray}}
\ncom{\nno}{\nonumber}
\ncom{\non}{\nonumber}
\ncom{\ds}{\displaystyle}
\ncom{\half}{\frac{1}{2}}
\ncom{\mbx}{\makebox{.25cm}}
\ncom{\hs}{\mbox{\hspace{.25cm}}}
\ncom{\rar}{\rightarrow}
\ncom{\Rar}{\Rightarrow}
\ncom{\noin}{\noindent}
\ncom{\bc}{\begin{center}}
\ncom{\ec}{\end{center}}
\ncom{\sz}{\scriptsize}
\ncom{\rf}{\ref}
\ncom{\s}{\sqrt{2}}
\ncom{\sgm}{\sigma}
\ncom{\Sgm}{\Sigma}
\ncom{\psgm}{\sigma^{\prime}}
\ncom{\dt}{\delta}
\ncom{\Dt}{\Delta}
\ncom{\lmd}{\lambda}
\ncom{\Lmd}{\Lambda}
\ncom{\Th}{\Theta}
\ncom{\e}{\eta}
\ncom{\eps}{\epsilon}
\ncom{\pcc}{\stackrel{P}{>}}
\ncom{\lp}{\stackrel{L_{p}}{>}}
\ncom{\dist}{{\rm\,dist}}
\ncom{\sspan}{{\rm\,span}}
\ncom{\re}{{\rm Re\,}}
\ncom{\im}{{\rm Im\,}}
\ncom{\sgn}{{\rm sgn\,}}
\ncom{\ba}{\begin{array}}
\ncom{\ea}{\end{array}}
\ncom{\hone}{\mbox{\hspace{1em}}}
\ncom{\htwo}{\mbox{\hspace{2em}}}
\ncom{\hthree}{\mbox{\hspace{3em}}}
\ncom{\hfour}{\mbox{\hspace{4em}}}
\ncom{\vone}{\vskip 2ex}
\ncom{\vtwo}{\vskip 4ex}
\ncom{\vonee}{\vskip 1.5ex}
\ncom{\vthree}{\vskip 6ex}
\ncom{\vfour}{\vspace*{8ex}}
\ncom{\norm}{\|\;\;\|}
\ncom{\integ}[4]{\int_{#1}^{#2}\,{#3}\,d{#4}}
\ncom{\vspan}[1]{{{\rm\,span}\{ #1 \}}}
\ncom{\dm}[1]{ {\displaystyle{#1} } }
\ncom{\ri}[1]{{#1} \index{#1}}

\newtheorem{theorem}{\bf Theorem}[section]
\newtheorem{remark}{\bf Remark}[section]

\newtheorem{lemma}{Lemma}[section]
\newtheorem{corollary}{Corollary}[section]

\newtheoremstyle
    {remarkstyle}
    {}
    {11pt}
    {}
    {}
    {\bfseries}
    {:}
    {     }
    {\thmname{#1} \thmnumber{#2} }

\theoremstyle{remarkstyle}

\usepackage[backend=bibtex,firstinits=true, url=false, isbn=false]{biblatex}
\def\eps{\varepsilon}

\begin{document}
\title{The Generalized $K$-Wright Function and Marichev-Saigo-Maeda Fractional Operators}
\author[Kuldeep Kumar Kataria]{K. K. Kataria}
\address{Kuldeep Kumar Kataria, Department of Mathematics,
 Indian Institute of Technology Bombay, Powai, Mumbai 400076, INDIA.}
 \email{kulkat@math.iitb.ac.in}
\author{P. Vellaisamy}
\address{P. Vellaisamy, Department of Mathematics,
 Indian Institute of Technology Bombay, Powai, Mumbai 400076, INDIA.}
 \email{pv@math.iitb.ac.in}
\thanks{The research of K. K. Kataria was supported by UGC, Govt. of India.}
\subjclass[2010]{Primary : 26A33, 33C20; Secondary : 33C65, 33C05}
\keywords{Generalized $K$-Wright function; Marichev-Saigo-Maeda fractional operators; Appell function.}
\begin{abstract}
In this paper, the generalized fractional operators involving Appell's function $F_3$ in the kernel due to Marichev-Saigo-Maeda are applied to the generalized $K$-Wright function. These fractional operators when applied to power multipliers of the generalized $K$-Wright function ${}_{p}\Psi^k_q$ yields a higher ordered generalized $K$-Wright function, namely, ${}_{p+3}\Psi^k_{q+3}$. The Caputo-type modification of Marichev-Saigo-Maeda fractional differentiation is introduced and the corresponding assertions for Saigo and Erd\'elyi-Kober fractional operators are also presented. The results derived in this paper generalize several recent results in the theory of special functions. 
\end{abstract}

\maketitle
\section{Introduction}
Throughout this paper $\mathbb{R}$ and $\mathbb{C}$ denote the sets of real and complex numbers respectively. Also $\mathbb{R}^+=(0,\infty)$, $\mathbb{N}_0=\{0,1,\ldots\}$ and $\mathbb{Z}^-=\{-1,-2,\ldots\}$.
Let $\alpha_i,\beta_j\in\mathbb{R}\setminus\{0\}$ and $a_i,b_
j\in\mathbb{C}$, $i=1,2,\ldots,p;\ j=1,2,\ldots,q$. Then the Generalized Wright function is defined for $z\in\mathbb{C}$ by the series
\begin{equation}\label{1.1}
{}_{p}\Psi_q(z)={}_{p}\Psi_q\Bigg[\left.
\begin{matrix}
    (a_i,\alpha_i)_{1,p}\\ 
    (b_j,\beta_j)_{1,q}
  \end{matrix}
\right|z\Bigg]
:=\sum_{n=0}^{\infty}\frac{\prod_{i=1}^{p}\Gamma(a_i+n\alpha_i)}{\prod_{j=1}^{q}\Gamma(b_j+n\beta_j)}\frac{z^n}{n!},
\end{equation}
where $\Gamma(z)$ is the Euler gamma function \cite{Erdelyi1953}. This function was introduced by Wright \cite{Wright1935286} and conditions for its existence together with its representation in terms of the Mellin-Barnes integral and of the \textit{H}-function were established by Kilbas \textit{et al.} (for details see \cite{Kilbas2002437}). Recently, Gehlot and Prajapati \cite{Gehlot201381}, introduced the following generalized $K$-Wright function defined in terms of $k$-gamma function by the series
\begin{equation}\label{1.2}
{}_{p}\Psi_q^k(z)={}_{p}\Psi_q^k\Bigg[\left.
\begin{matrix}
    (a_i,\alpha_i)_{1,p}\\ 
    (b_j,\beta_j)_{1,q}
  \end{matrix}
\right|z\Bigg]
:=\sum_{n=0}^{\infty}\frac{\prod_{i=1}^{p}\Gamma_k(a_i+n\alpha_i)}{\prod_{j=1}^{q}\Gamma_k(b_j+n\beta_j)}\frac{z^n}{n!},
\end{equation}
where $k\in\mathbb{R}^+$ and $(a_i+n\alpha_i),(b_j+n\beta_j)\in\mathbb{C}\setminus k\mathbb{Z}^-\ \forall\ n\in\mathbb{N}_0$. Here, $\Gamma_k(z)$ is the generalized $k$-gamma function \cite{Diaz2007179} defined for $z\in\mathbb{C}\setminus k\mathbb{Z}^-$ and $k\in\mathbb{R}^+$ by
\begin{equation*}\label{1.3}
\Gamma_k(z)=\underset{n\rightarrow\infty}{\lim}\frac{n!k^n(nk)^{\frac{z}{k}-1}}{(z)_{n,k}},
\end{equation*}
and $(z)_{n,k}$ is the $k$-Pochhammer symbol \cite{Diaz2007179} defined for complex $z\in\mathbb{C}$ and $k\in\mathbb{R}$ by
\begin{equation}\label{1.4}
(z)_{n,k}=\left\{
	\begin{array}{ll}
	    1 & \mbox{if } n=0,\\
		z(z+k)(z+2k)\ldots(z+(n-1)k)  & \mbox{if } n\in\mathbb{N}.
	\end{array}
\right.
\end{equation}
The integral representation of $\Gamma_k(z)$ is defined for $z\in\mathbb{C},\ \operatorname{Re}(z)>0$ and $k\in\mathbb{R}^+$  by
\begin{equation*}\label{1.5}
\Gamma_k(z)=\int_0^{\infty}t^{z-1}e^{-\frac{t^k}{k}}\,dt,
\end{equation*}
due to which the following identity holds (see \cite{Diaz2007179}):
\begin{equation}\label{1.6}
\Gamma_k(z)=k^{\frac{z}{k}-1}\Gamma\left(\frac{z}{k}\right).
\end{equation}
For $k=1$, the generalized $K$-Wright function reduces to the generalized Wright function. Wright function \cite{Erdelyi1954}, Bessel-Maitland function (see \cite{Kiryakova1994}, \cite{Marichev1983}) and generalized Mittag-Leffler function \cite{Kilbas2002437} are some other particular cases. The asymptotic behavior of the generalized Wright function ${}_{p}\Psi_q^k(z)$ for large values of argument of $z$ was studied by Fox \cite{Fox389} and Wright \cite{Wright1935286},\cite{Wright1940423},\cite{Wright1940389} under the condition
\begin{equation*}\label{1.7}
\sum_{j=1}^q\beta_j-\sum_{i=1}^p\alpha_i>-1.
\end{equation*}
Saigo \cite{Saigo1978135} introduced the fractional integral and differential operators involving Gauss hypergeometric function as the kernel, which are interesting generalizations of classical Riemann-Liouville and Erd\'elyi-Kober fractional operators (for details see \cite{Kilbas2006}). For $\alpha,\beta,\gamma\in\mathbb{C}$ and $x\in\mathbb{R}^+$ with $\operatorname{Re}(\alpha)>0$, the left- and right-hand sided generalized fractional integral operators associated with Gauss hypergeometric function are defined by
\begin{equation}\label{1.8}
\left(I_{0+}^{\alpha,\beta,\gamma}f\right)(x)=\frac{x^{-\alpha-\beta}}{\Gamma(\alpha)}\int_0^x(x-t)^{\alpha-1}{}_{2}F_1\left(\alpha+\beta,-\gamma;\alpha;1-\frac{t}{x}\right)f(t)\,dt
\end{equation}
and
\begin{equation}\label{1.9}
\left(I_{-}^{\alpha,\beta,\gamma}f\right)(x)=\frac{1}{\Gamma(\alpha)}\int_x^{\infty}(t-x)^{\alpha-1}t^{-\alpha-\beta}{}_{2}F_1\left(\alpha+\beta,-\gamma;\alpha;1-\frac{x}{t}\right)f(t)\,dt,
\end{equation}
respectively. Here ${}_{2}F_1(a,b;c;z)$ is the Gauss hypergeometric function \cite{Kilbas2006} defined for $z\in\mathbb{C},\ |z|<1$ and $\alpha,\beta\in\mathbb{C},\ \gamma\in\mathbb{C}\setminus\mathbb{Z}_0^-$ in the unit disk as the sum of the hypergeometric series
\begin{equation*}\label{1.12}
{}_{2}F_1(\alpha,\beta;\gamma;z)=\sum_{n=0}^{\infty}\frac{(\alpha)_{n}(\beta)_{n}}{(\gamma)_{n}}\frac{z^n}{n!}.
\end{equation*}
where $(z)_{n}$ is the standard Pochhammer symbol which in view of (\ref{1.4}) is equal to $(z)_{n,1}$. The corresponding fractional differential operators have their respective forms as
\begin{equation}\label{1.10}
\left(D_{0+}^{\alpha,\beta,\gamma}f\right)(x)=\left(\frac{d}{dx}\right)^{[\operatorname{Re}(\alpha)]+1}\left(I_{0+}^{-\alpha+[\operatorname{Re}(\alpha)]+1,-\beta-[\operatorname{Re}(\alpha)]-1,\alpha+\gamma-[\operatorname{Re}(\alpha)]-1}f\right)(x)
\end{equation}
and
\begin{equation}\label{1.11}
\left(D_{-}^{\alpha,\beta,\gamma}f\right)(x)=\left(-\frac{d}{dx}\right)^{[\operatorname{Re}(\alpha)]+1}\left(I_{-}^{-\alpha+[\operatorname{Re}(\alpha)]+1,-\beta-[\operatorname{Re}(\alpha)]-1,\alpha+\gamma}f\right)(x),
\end{equation}
where $[\operatorname{Re}(\alpha)]$ denotes the integer part of $\operatorname{Re}(\alpha)$. For $\beta=-\alpha$ and $\beta=0$ in (\ref{1.8})-(\ref{1.11}), we get the corresponding Riemann-Liouville and Erd\'elyi-Kober fractional operators respectively.

\noindent The special function $F_3$, called third Appell function \cite{Erdelyi1953}, \cite{Prudnikov1990} (also known as Horn function), is defined by
\begin{equation*}\label{1.17}
F_3(\alpha,\alpha',\beta,\beta';\gamma;x;y)=\sum_{m,n=0}^{\infty}\frac{(\alpha)_{m}(\alpha')_{n}(\beta)_{m}(\beta')_{n}}{(\gamma)_{m+n}}\frac{x^my^n}{m!n!},
\end{equation*}
such that $\max\{|x|,|y|\}<1$. It is related to Gauss hypergeometric function as
\begin{equation*}\label{1.18}
F_3(\alpha,\gamma-\alpha,\beta,\gamma-\beta;\gamma;x;y)={}_{2}F_1\left(\alpha,\beta,\gamma;x+y-xy\right).
\end{equation*}
A generalization of Saigo operators was introduced by Marichev \cite{Marichev1974128} as Mellin type convolution operators with Appell function as the kernel, which were later extended and studied by Saigo and Maeda \cite{Saigo1998}. For $\alpha,\alpha',\beta,\beta',\gamma\in\mathbb{C}$ and $x\in\mathbb{R}^+$ with $\operatorname{Re}(\gamma)>0$, the left- and right-hand sided Marichev-Saigo-Maeda fractional integral operators associated with third Appell function are defined by
\begin{equation}\label{1.13}
\left(I_{0+}^{\alpha,\alpha',\beta,\beta',\gamma}f\right)(x)=\frac{x^{-\alpha}}{\Gamma(\gamma)}\int_0^x(x-t)^{\gamma-1}t^{-\alpha'}F_3\left(\alpha,\alpha',\beta,\beta';\gamma;1-\frac{t}{x},1-\frac{x}{t}\right)f(t)\,dt
\end{equation}
and
\begin{equation}\label{1.14}
\left(I_{-}^{\alpha,\alpha',\beta,\beta',\gamma}f\right)(x)=\frac{x^{-\alpha'}}{\Gamma(\gamma)}\int_x^{\infty}(t-x)^{\gamma-1}t^{-\alpha}F_3\left(\alpha,\alpha',\beta,\beta';\gamma;1-\frac{x}{t},1-\frac{t}{x}\right)f(t)\,dt,
\end{equation}
respectively. The corresponding fractional differential operators have their respective forms as
\begin{equation}\label{1.15}
\left(D_{0+}^{\alpha,\alpha',\beta,\beta',\gamma}f\right)(x)=\left(\frac{d}{dx}\right)^{[\operatorname{Re}(\gamma)]+1}\left(I_{0+}^{-\alpha',-\alpha,-\beta'+[\operatorname{Re}(\gamma)]+1,-\beta,-\gamma+[\operatorname{Re}(\gamma)]+1}f\right)(x)
\end{equation}
and
\begin{equation}\label{1.16}
\left(D_{-}^{\alpha,\alpha',\beta,\beta',\gamma}f\right)(x)=\left(-\frac{d}{dx}\right)^{[\operatorname{Re}(\gamma)]+1}\left(I_{-}^{-\alpha',-\alpha,-\beta',-\beta+[\operatorname{Re}(\gamma)]+1,-\gamma+[\operatorname{Re}(\gamma)]+1}f\right)(x).
\end{equation}
These operators generalize the well-known operators such as the Riemann-Liouville , Weyl, Erd´elyi-Kober and Saigo operators. Marichev-Saigo-Maeda (MSM) fractional operators (\ref{1.13})-(\ref{1.16}) are connected to Saigo operators (\ref{1.8})-(\ref{1.11}) as follows:
\begin{equation}\label{1.19}
\left(I_{0+}^{\alpha,0,\beta,\beta',\gamma}f\right)(x)=\left(I_{0+}^{\gamma,\alpha-\gamma,-\beta}f\right)(x),\ \ \ \ 
\left(I_{-}^{\alpha,0,\beta,\beta',\gamma}f\right)(x)=\left(I_{-}^{\gamma,\alpha-\gamma,-\beta}f\right)(x)
\end{equation}
and
\begin{equation}\label{1.20}
\left(D_{0+}^{0,\alpha',\beta,\beta',\gamma}f\right)(x)=\left(D_{0+}^{\gamma,\alpha'-\gamma,\beta'-\gamma}f\right)(x),\ \ \ \ 
\left(D_{-}^{0,\alpha',\beta,\beta',\gamma}f\right)(x)=\left(D_{-}^{\gamma,\alpha'-\gamma,\beta'-\gamma}f\right)(x).
\end{equation}
The paper is organized as follow. In section 2 some preliminary results are stated. In sections 3 and 4, we consider the MSM fractional integral and differential transform of the generalized $K$-Wright function. In section 5, Caputo type modification of the MSM fractional differential operators is introduced. The corresponding results for the Saigo and Erd\'elyi-Kober fractional operators are mentioned as corollaries. 
\section{Preliminaries}
\setcounter{equation}{0}
In this section, we present the conditions for the existence of the generalized $K$-Wright function ${}_{p}\Psi^k_q(z)$. We use the following notations:
\begin{equation}\label{2.1}
\Delta=\sum_{j=1}^q\left(\frac{\beta_j}{k}\right)-\sum_{i=1}^p\left(\frac{\alpha_i}{k}\right),
\end{equation}
\begin{equation*}
\delta=\prod_{i=1}^p\left|\frac{\alpha_i}{k}\right|^{-\frac{\alpha_i}{k}}\prod_{j=1}^q\left|\frac{\beta_j}{k}\right|^{\frac{\beta_j}{k}}\ \ \mathrm{and}\ \ \ 
\mu=\sum_{j=1}^q\left(\frac{b_j}{k}\right)-\sum_{i=1}^p\left(\frac{a_i}{k}\right)+\frac{p-q}{2}.\nonumber
\end{equation*}
The following result is proved in \cite{Gehlot201381}.
\begin{theorem}\label{t2.1}
Let $k\in\mathbb{R}^+$, $a_i,b_j\in\mathbb{C}$ and $\alpha_i,\beta_j\in\mathbb{R}$ $(\alpha_i,\beta_j\neq0\ \forall\ i=1,2,\ldots p$ and $j=1,2,\ldots q)$ such that $(a_i+n\alpha_i),(b_j+n\beta_j)\in\mathbb{C}\setminus k\mathbb{Z}^-\ \forall\ n\in\mathbb{N}_0$.
\begin{itemize}
\item[(a)] If $\Delta>-1$, then the series (\ref{1.2}) is absolutely convergent for all $z\in\mathbb{C}$ and ${}_{p}\Psi^k_q(z)$ is an entire function of $z$.
\item[(b)] If $\Delta=-1$, then the series (\ref{1.2}) is absolutely convergent for $|z|<\delta$, and for $|z|=\delta$ and $\operatorname{Re}(\mu)>\frac{1}{2}$.
\end{itemize}
\end{theorem}
\noindent The following are well known results for MSM integral operators of power functions (see \cite{Prudnikov1990}, \cite{Saigo1998}).
\begin{lemma}\label{l2.1}
Let $\alpha,\alpha',\beta,\beta',\gamma,\rho\in\mathbb{C}$ such that $\operatorname{Re}(\gamma)>0$.
\begin{itemize}
\item[(a)] If $\operatorname{Re}(\rho)>$ $\max\{0,\operatorname{Re}(\alpha'-\beta'),\operatorname{Re}(\alpha+\alpha'+\beta-\gamma)\}$, then
\begin{equation}\label{2.4}
\left(I_{0+}^{\alpha,\alpha',\beta,\beta',\gamma}t^{\rho-1}\right)(x)=\frac{\Gamma(\rho)\Gamma(-\alpha'+\beta'+\rho)\Gamma(-\alpha-\alpha'-\beta+\gamma+\rho)}{\Gamma(\beta'+\rho)\Gamma(-\alpha-\alpha'+\gamma+\rho)\Gamma(-\alpha'-\beta+\gamma+\rho)}x^{-\alpha-\alpha'+\gamma+\rho-1}.
\end{equation}
\item[(b)] If $\operatorname{Re}(\rho)>$ $\max\{\operatorname{Re}(\beta),\operatorname{Re}(-\alpha-\alpha'+\gamma),\operatorname{Re}(-\alpha-\beta'+\gamma)\}$, then
\begin{equation}\label{2.5}
\left(I_{-}^{\alpha,\alpha',\beta,\beta',\gamma}t^{-\rho}\right)(x)=\frac{\Gamma(-\beta+\rho)\Gamma(\alpha+\alpha'-\gamma+\rho)\Gamma(\alpha+\beta'-\gamma+\rho)}{\Gamma(\rho)\Gamma(\alpha-\beta+\rho)\Gamma(\alpha+\alpha'+\beta'-\gamma+\rho)}x^{-\alpha-\alpha'+\gamma-\rho}.
\end{equation}
\end{itemize} 
\end{lemma}

\section{MSM fractional integration of ${}_{p}\Psi_q^k$}
\setcounter{equation}{0}
\noindent In this section, we consider the Marichev-Saigo-Maeda fractional integration of the generalized $K$-Wright function. First theorem deals with the left-hand sided MSM fractional integration of ${}_{p}\Psi_q^k$.
\begin{theorem}\label{t3.1}
Let $\alpha,\alpha',\beta,\beta',\gamma,\rho\in\mathbb{C}$ and $k\in\mathbb{R}^+$ be such that $\operatorname{Re}(\gamma)>0$, $\operatorname{Re}\left(\frac{\rho}{k}\right)>\max\{0,$ $\operatorname{Re}(\alpha'-\beta'),\operatorname{Re}(\alpha+\alpha'+\beta-\gamma)\}$. Also, let $a\in\mathbb{C}$ and $\mu>0$. If $\Delta>-1$ in (\ref{2.1}), then for $x>0$
\begin{eqnarray}\label{3.1}
&&\left(I_{0+}^{\alpha,\alpha',\beta,\beta',\gamma}\left(t^{\frac{\rho}{k}-1}{}_{p}\Psi_q^k\Bigg[\left.
\begin{matrix}
    (a_i,\alpha_i)_{1,p}\\ 
    (b_j,\beta_j)_{1,q}
  \end{matrix}
\right|at^{\frac{\mu}{k}}\Bigg]\right)\right)(x)\nonumber\\
&=&k^{\gamma}x^{-\alpha-\alpha'+\gamma+\frac{\rho}{k}-1}{}_{p+3}\Psi_{q+3}^k\Bigg[
\begin{matrix}
    (a_i,\alpha_i)_{1,p}&(\rho,\mu)\\ 
    (b_j,\beta_j)_{1,q}&(k\beta'+\rho,\mu)
  \end{matrix}\nonumber\\
& &\left.
\begin{matrix}
   (-k\alpha'+k\beta'+\rho,\mu)&(-k\alpha-k\alpha'-k\beta+k\gamma+\rho,\mu)\\ 
   (-k\alpha-k\alpha'+k\gamma+\rho,\mu)&(-k\alpha'-k\beta+k\gamma+\rho,\mu)
  \end{matrix}
\right|ax^{\frac{\mu}{k}}\Bigg].
\end{eqnarray}
\end{theorem}
\begin{proof}
From (\ref{1.2}), we have
\begin{equation}\label{3.2}
\mathrm{lhs\ of\ (\ref{3.1})} =\left(I_{0+}^{\alpha,\alpha',\beta,\beta',\gamma}\left(t^{\frac{\rho}{k}-1}\sum_{n=0}^{\infty}\frac{\prod_{i=1}^{p}\Gamma_k(a_i+n\alpha_i)}{\prod_{j=1}^{q}\Gamma_k(b_j+n\beta_j)}\frac{(at^{\frac{\mu}{k}})^n}{n!}\right)\right)(x).
\end{equation}
Interchanging the order of integration and summation, which is justified by the absolute convergence of the integral and the uniform convergence of the series involved, we get the rhs of (\ref{3.2}) as 
\begin{eqnarray*}
&&\sum_{n=0}^{\infty}\frac{\prod_{i=1}^{p}\Gamma_k(a_i+n\alpha_i)}{\prod_{j=1}^{q}\Gamma_k(b_j+n\beta_j)}\frac{a^n}{n!}\left(I_{0+}^{\alpha,\alpha',\beta,\beta',\gamma}t^{\frac{\rho}{k}+\frac{n\mu}{k}-1}\right)(x)\\
&=&\sum_{n=0}^{\infty}\frac{\prod_{i=1}^{p}\Gamma_k(a_i+n\alpha_i)}{\prod_{j=1}^{q}\Gamma_k(b_j+n\beta_j)}\frac{a^n}{n!}\frac{\Gamma\left(\frac{\rho}{k}+\frac{n\mu}{k}\right)}{\Gamma\left(\beta'+\frac{\rho}{k}+\frac{n\mu}{k}\right)}\\
& &\times\ \frac{\Gamma\left(-\alpha'+\beta'+\frac{\rho}{k}+\frac{n\mu}{k}\right)\Gamma\left(-\alpha-\alpha'-\beta+\gamma+\frac{\rho}{k}+\frac{n\mu}{k}\right)}{\Gamma\left(-\alpha-\alpha'+\gamma+\frac{\rho}{k}+\frac{n\mu}{k}\right)\Gamma\left(-\alpha'-\beta+\gamma+\frac{\rho}{k}+\frac{n\mu}{k}\right)}x^{-\alpha-\alpha'+\gamma+\frac{\rho}{k}+\frac{n\mu}{k}-1}\\
&=&k^{\gamma}x^{-\alpha-\alpha'+\gamma+\frac{\rho}{k}-1}\sum_{n=0}^{\infty}\frac{\prod_{i=1}^{p}\Gamma_k(a_i+n\alpha_i)}{\prod_{j=1}^{q}\Gamma_k(b_j+n\beta_j)}\frac{\Gamma_k(\rho+n\mu)}{\Gamma_k(k\beta'+\rho+n\mu)}\\
& &\times\ \frac{\Gamma(-k\alpha'+k\beta'+\rho+n\mu)\Gamma_k(-k\alpha-k\alpha'-k\beta+k\gamma+\rho+n\mu)}{\Gamma_k(-k\alpha-k\alpha'+k\gamma+\rho+n\mu)\Gamma(-k\alpha'-k\beta+k\gamma+\rho+n\mu)}\frac{(ax^{\frac{\mu}{k}})^n}{n!},
\end{eqnarray*}
where we have used (\ref{2.4}) and (\ref{1.6}). Finally, by using (\ref{1.2}), the proof is complete.
\end{proof}
\noindent In view of (\ref{1.19}), we have the following result for Saigo operators.
\begin{corollary}\label{c3.1}
Let $\alpha,\beta,\gamma,\rho\in\mathbb{C}$ and $k\in\mathbb{R}^+$ be such that $\operatorname{Re}(\alpha)>0$, $\operatorname{Re}\left(\frac{\rho}{k}\right)>\max\{0,\operatorname{Re}(\beta-\gamma)\}$, and also let $a\in\mathbb{C},\ \mu>0$. If $\Delta>-1$ in (\ref{2.1}), then the left-hand sided generalized fractional integration $I_{0+}^{\alpha,\beta,\gamma}$ of ${}_{p}\Psi_q^k$ is given for $x>0$ by
\begin{eqnarray*}\label{3.3}
&&\left(I_{0+}^{\alpha,\beta,\gamma}\left(t^{\frac{\rho}{k}-1}{}_{p}\Psi_q^k\Bigg[\left.
\begin{matrix}
    (a_i,\alpha_i)_{1,p}\\ 
    (b_j,\beta_j)_{1,q}
  \end{matrix}
\right|at^{\frac{\mu}{k}}\Bigg]\right)\right)(x)\nonumber\\
&&=k^{\alpha}x^{-\beta+\frac{\rho}{k}-1}{}_{p+2}\Psi_{q+2}^k\Bigg[\left.
\begin{matrix}
    (a_i,\alpha_i)_{1,p}&(\rho,\mu)&(-k\beta+k\gamma+\rho,\mu)\\ 
    (b_j,\beta_j)_{1,q}&(-k\beta+\rho,\mu)&(k\alpha+k\gamma+\rho,\mu)
  \end{matrix}
\right|ax^{\frac{\mu}{k}}\Bigg].
\end{eqnarray*}
\end{corollary}
\noindent Further the corresponding result for Erd\'elyi-Kober fractional integral (see \cite{Kilbas2006}) is as follows.
\begin{corollary}\label{c3.3}
Let $\alpha,\gamma,\rho\in\mathbb{C}$ and $k\in\mathbb{R}^+$ be such that $\operatorname{Re}(\alpha)>0$, $\operatorname{Re}\left(\frac{\rho}{k}\right)>\max\{0,\operatorname{Re}(-\gamma)\}$, and also let $a\in\mathbb{C},\ \mu>0$. If $\Delta>-1$ in (\ref{2.1}), then the left-hand sided Erd\'elyi-Kober fractional integration $I_{\gamma,\alpha}^{+}$ $(=I_{0+}^{\alpha,0,\gamma})$ of ${}_{p}\Psi_q^k$ is given for $x>0$ by
\begin{eqnarray*}\label{3.5}
&&\left(I_{\gamma,\alpha}^{+}\left(t^{\frac{\rho}{k}-1}{}_{p}\Psi_q^k\Bigg[\left.
\begin{matrix}
    (a_i,\alpha_i)_{1,p}\\ 
    (b_j,\beta_j)_{1,q}
  \end{matrix}
\right|at^{\frac{\mu}{k}}\Bigg]\right)\right)(x)\nonumber\\
&&=k^{\alpha}x^{\frac{\rho}{k}-1}{}_{p+1}\Psi_{q+1}^k\Bigg[\left.
\begin{matrix}
    (a_i,\alpha_i)_{1,p}&(k\gamma+\rho,\mu)\\ 
    (b_j,\beta_j)_{1,q}&(k\alpha+k\gamma+\rho,\mu)
  \end{matrix}
\right|ax^{\frac{\mu}{k}}\Bigg].
\end{eqnarray*}
\end{corollary}
\noindent Next we consider the right-hand sided MSM fractional integration of ${}_{p}\Psi_q^k$.
\begin{theorem}\label{t3.2}
Let $\alpha,\alpha',\beta,\beta',\gamma,\rho\in\mathbb{C}$ and $k\in\mathbb{R}^+$ be such that $\operatorname{Re}(\gamma)>0$, $\operatorname{Re}\left(\frac{\rho}{k}\right)>\max\{\operatorname{Re}(\beta),$ $\operatorname{Re}(-\alpha-\alpha'+\gamma),\operatorname{Re}(-\alpha-\beta'+\gamma)\}$. Also, let $a\in\mathbb{C}$ and $\mu>0$. If $\Delta>-1$ in (\ref{2.1}), then
\begin{eqnarray}\label{3.6}
&&\left(I_{-}^{\alpha,\alpha',\beta,\beta',\gamma}\left(t^{-\frac{\rho}{k}}{}_{p}\Psi_q^k\Bigg[\left.
\begin{matrix}
    (a_i,\alpha_i)_{1,p}\\ 
    (b_j,\beta_j)_{1,q}
  \end{matrix}
\right|at^{-\frac{\mu}{k}}\Bigg]\right)\right)(x)\nonumber\\
&=&k^{\gamma}x^{-\alpha-\alpha'+\gamma-\frac{\rho}{k}}{}_{p+3}\Psi_{q+3}^k\Bigg[
\begin{matrix}
    (a_i,\alpha_i)_{1,p}&(-k\beta+\rho,\mu)\\ 
    (b_j,\beta_j)_{1,q}&(\rho,\mu)
  \end{matrix}\nonumber\\
& &\left.
\begin{matrix}
   (k\alpha+k\alpha'-k\gamma+\rho,\mu)&(k\alpha+k\beta'-k\gamma+\rho,\mu)\\ 
   (k\alpha-k\beta+\rho,\mu)&(k\alpha+k\alpha'+k\beta'-k\gamma+\rho,\mu)
  \end{matrix}
\right|ax^{-\frac{\mu}{k}}\Bigg],
\end{eqnarray}
for $x>0$.
\end{theorem}
\begin{proof}
Using (\ref{1.2}), lhs of (\ref{3.6}) equals
\begin{equation}\label{3.7}
\left(I_{-}^{\alpha,\alpha',\beta,\beta',\gamma}\left(t^{-\frac{\rho}{k}}\sum_{n=0}^{\infty}\frac{\prod_{i=1}^{p}\Gamma_k(a_i+n\alpha_i)}{\prod_{j=1}^{q}\Gamma_k(b_j+n\beta_j)}\frac{(at^{-\frac{\mu}{k}})^n}{n!}\right)\right)(x).
\end{equation}
By a term by term integration of the above series and using (\ref{2.5}) and (\ref{1.6}), (\ref{3.7}) reduces to
\begin{eqnarray*}
&&\sum_{n=0}^{\infty}\frac{\prod_{i=1}^{p}\Gamma_k(a_i+n\alpha_i)}{\prod_{j=1}^{q}\Gamma_k(b_j+n\beta_j)}\frac{a^n}{n!}\left(I_{-}^{\alpha,\alpha',\beta,\beta',\gamma}t^{-\left(\frac{\rho}{k}+\frac{n\mu}{k}\right)}\right)(x)\\
&=&\sum_{n=0}^{\infty}\frac{\prod_{i=1}^{p}\Gamma_k(a_i+n\alpha_i)}{\prod_{j=1}^{q}\Gamma_k(b_j+n\beta_j)}\frac{a^n}{n!}\frac{\Gamma\left(-\beta+\frac{\rho}{k}+\frac{n\mu}{k}\right)}{\Gamma\left(\frac{\rho}{k}+\frac{n\mu}{k}\right)}\\
& &\times\ \frac{\Gamma\left(\alpha+\alpha'-\gamma+\frac{\rho}{k}+\frac{n\mu}{k}\right)\Gamma\left(\alpha+\beta'-\gamma+\frac{\rho}{k}+\frac{n\mu}{k}\right)}{\Gamma\left(\alpha-\beta+\frac{\rho}{k}+\frac{n\mu}{k}\right)\Gamma\left(\alpha+\alpha'+\beta'-\gamma+\frac{\rho}{k}+\frac{n\mu}{k}\right)}x^{-\alpha-\alpha'+\gamma-\frac{\rho}{k}-\frac{n\mu}{k}}\\
&=&k^{\gamma}x^{-\alpha-\alpha'+\gamma-\frac{\rho}{k}}\sum_{n=0}^{\infty}\frac{\prod_{i=1}^{p}\Gamma_k(a_i+n\alpha_i)}{\prod_{j=1}^{q}\Gamma_k(b_j+n\beta_j)}\frac{\Gamma_k(-k\beta+\rho+n\mu)}{\Gamma_k(\rho+n\mu)}\\
& &\times\ \frac{\Gamma_k(k\alpha+k\alpha'-k\gamma+\rho+n\mu)\Gamma_k(k\alpha+k\beta'-k\gamma+\rho+n\mu)}{\Gamma_k(k\alpha-k\beta+\rho+n\mu)\Gamma_k(k\alpha+k\alpha'+k\beta'-k\gamma+\rho+n\mu)}\frac{(ax^{-\frac{\mu}{k}})^n}{n!}.
\end{eqnarray*}
The result follows from (\ref{1.2}).
\end{proof}
\noindent The corresponding results for Saigo and Erd\'elyi-Kober fractional integration are as follows.
\begin{corollary}\label{c3.4}
Let $\alpha,\beta,\gamma,\rho\in\mathbb{C}$ and $k\in\mathbb{R}^+$ be such that $\operatorname{Re}(\alpha)>0$, $\operatorname{Re}\left(\frac{\rho}{k}\right)>\max\{\operatorname{Re}(-\beta),$ $\operatorname{Re}(-\gamma)\}$, and also let $a\in\mathbb{C},\ \mu>0$. If $\Delta>-1$ in (\ref{2.1}), then the right-hand sided generalized fractional integration $I_{-}^{\alpha,\beta,\gamma}$ of ${}_{p}\Psi_q^k$ is given for $x>0$ by
\begin{eqnarray*}\label{3.8}
&&\left(I_{-}^{\alpha,\beta,\gamma}\left(t^{-\frac{\rho}{k}}{}_{p}\Psi_q^k\Bigg[\left.
\begin{matrix}
    (a_i,\alpha_i)_{1,p}\\ 
    (b_j,\beta_j)_{1,q}
  \end{matrix}
\right|at^{-\frac{\mu}{k}}\Bigg]\right)\right)(x)\nonumber\\
&&=k^{\alpha}x^{-\beta-\frac{\rho}{k}}{}_{p+2}\Psi_{q+2}^k\Bigg[\left.
\begin{matrix}
    (a_i,\alpha_i)_{1,p}&(k\beta+\rho,\mu)&(k\gamma+\rho,\mu)\\ 
    (b_j,\beta_j)_{1,q}&(\rho,\mu)&(k\alpha+k\beta+k\gamma+\rho,\mu)
  \end{matrix}
\right|ax^{-\frac{\mu}{k}}\Bigg].
\end{eqnarray*}
\end{corollary}
\begin{corollary}\label{c3.6}
Let $\alpha,\gamma,\rho\in\mathbb{C}$ and $k\in\mathbb{R}^+$ be such that $\operatorname{Re}(\alpha)>0$, $\operatorname{Re}\left(\frac{\rho}{k}\right)>\max\{0,\operatorname{Re}(-\gamma)\}$, and also let $a\in\mathbb{C},\ \mu>0$. If $\Delta>-1$ in (\ref{2.1}), then the right-hand sided Erd\'elyi-Kober fractional integration $K_{\gamma,\alpha}^{-}$ $(=I_{-}^{\alpha,0,\gamma})$ of ${}_{p}\Psi_q^k$ is given for $x>0$ by
\begin{eqnarray*}\label{3.10}
&&\left(K_{\gamma,\alpha}^{-}\left(t^{-\frac{\rho}{k}}{}_{p}\Psi_q^k\Bigg[\left.
\begin{matrix}
    (a_i,\alpha_i)_{1,p}\\ 
    (b_j,\beta_j)_{1,q}
  \end{matrix}
\right|at^{-\frac{\mu}{k}}\Bigg]\right)\right)(x)\nonumber\\
&&=k^{\alpha}x^{-\frac{\rho}{k}}{}_{p+1}\Psi_{q+1}^k\Bigg[\left.
\begin{matrix}
    (a_i,\alpha_i)_{1,p}&(k\gamma+\rho,\mu)\\ 
    (b_j,\beta_j)_{1,q}&(k\alpha+k\gamma+\rho,\mu)
  \end{matrix}
\right|ax^{-\frac{\mu}{k}}\Bigg].
\end{eqnarray*}
\end{corollary}
\begin{remark}
When $\beta=-\alpha$, Corollaries \ref{c3.1} and \ref{c3.4}, respectively yield left- and right-hand sided Riemann-Liouville fractional integration of ${}_{p}\Psi_q^k$ (see Gehlot and Prajapati \cite{Gehlot283}). Further, the case $k=1$, reduces to the results for generalized Wright function (see Kilbas \cite{Kilbas2004113}).
\end{remark} 

\section{MSM fractional differentiation of ${}_{p}\Psi_q^k$}
\setcounter{equation}{0}
To obtain the Marichev-Saigo-Maeda fractional differentiation of the generalized $K$-Wright function, we first prove the following lemma.
\begin{lemma}\label{ll4.1}
Let $\alpha,\alpha',\beta,\beta',\gamma,\rho\in\mathbb{C}$.
\begin{itemize}
\item[(a)] If $\operatorname{Re}(\rho)>$ $\max\{0,,\operatorname{Re}(-\alpha+\beta),\operatorname{Re}(-\alpha-\alpha'-\beta'+\gamma)\}$, then
\begin{equation}\label{l4.1}
\left(D_{0+}^{\alpha,\alpha',\beta,\beta',\gamma}t^{\rho-1}\right)(x)=\frac{\Gamma(\rho)\Gamma(-\beta+\alpha+\rho)\Gamma(\alpha+\alpha'+\beta'-\gamma+\rho)}{\Gamma(-\beta+\rho)\Gamma(\alpha+\alpha'-\gamma+\rho)\Gamma(\alpha+\beta'-\gamma+\rho)}x^{\alpha+\alpha'-\gamma+\rho-1}.
\end{equation}
\item[(b)] If $\operatorname{Re}(\rho)>$ $\max\{\operatorname{Re}(-\beta'),\operatorname{Re}(\alpha'+\beta-\gamma),\operatorname{Re}(\alpha+\alpha'-\gamma)+[\operatorname{Re}(\gamma)]+1\}$, then
\begin{equation}\label{l4.2}
\left(D_{-}^{\alpha,\alpha',\beta,\beta',\gamma}t^{-\rho}\right)(x)=\frac{\Gamma\left(\beta'+\rho\right)\Gamma\left(-\alpha-\alpha'+\gamma+\rho\right)\Gamma\left(-\alpha'-\beta+\gamma+\rho\right)}{\Gamma\left(\rho\right)\Gamma\left(-\alpha'+\beta'+\rho\right)\Gamma\left(-\alpha-\alpha'-\beta+\gamma+\rho\right)}x^{\alpha+\alpha'-\gamma-\rho}.
\end{equation}
\end{itemize} 
\end{lemma}
\begin{proof}
For convenience, let $m=[\operatorname{Re}(\gamma)]+1$. 
\begin{itemize}
\item[(a)] Using (\ref{1.15}) and (\ref{2.4}), lhs of (\ref{l4.1}) equals,
\begin{eqnarray*}
&&\left(\frac{d}{\,dx}\right)^m\left(I_{0+}^{-\alpha',-\alpha,-\beta'+m,-\beta,-\gamma+m}t^{\rho-1}\right)(x)\\
&=&\frac{d^m}{\,dx^m}\frac{\Gamma\left(\rho\right)\Gamma\left(\alpha-\beta+\rho\right)\Gamma\left(\alpha+\alpha'+\beta'-\gamma+\rho\right)}{\Gamma\left(-\beta+\rho\right)\Gamma\left(\alpha+\alpha'-\gamma+\rho+m\right)\Gamma\left(\alpha+\beta'-\gamma+\rho\right)}x^{\alpha+\alpha'-\gamma+\rho+m-1}\\
&=&\frac{\Gamma\left(\rho\right)\Gamma\left(\alpha-\beta+\rho\right)\Gamma\left(\alpha+\alpha'+\beta'-\gamma+\rho\right)}{\Gamma\left(-\beta+\rho\right)\Gamma\left(\alpha+\alpha'-\gamma+\rho+m\right)\Gamma\left(\alpha+\beta'-\gamma+\rho\right)}\frac{d^m}{\,dx^m}x^{\alpha+\alpha'-\gamma+\rho+m-1},
\end{eqnarray*}
which on differentiation yields (\ref{l4.1}).
\item[(b)] Using (\ref{1.16}) and (\ref{2.5}), lhs of (\ref{l4.2}) reduces to
\begin{eqnarray}\label{l4.7}
&&\left(-\frac{d}{\,dx}\right)^m\left(I_{-}^{-\alpha',-\alpha,-\beta',-\beta+m,-\gamma+m}t^{-\rho}\right)(x)\nonumber\\
&=&(-1)^m\frac{\Gamma\left(\beta'+\rho\right)\Gamma\left(-\alpha-\alpha'+\gamma+\rho-m\right)\Gamma\left(-\alpha'-\beta+\gamma+\rho\right)}{\Gamma\left(\rho\right)\Gamma\left(-\alpha'+\beta'+\rho\right)\Gamma\left(-\alpha-\alpha'-\beta+\gamma+\rho\right)}\frac{d^m}{\,dx^m}x^{\alpha+\alpha'-\gamma-\rho+m}.\nonumber\\
\end{eqnarray}
Now
\begin{equation}\label{l4.8}
\frac{d^m}{\,dx^m}x^{\alpha+\alpha'-\gamma-\rho+m}=\frac{\Gamma\left(\alpha+\alpha'-\gamma-\rho+m+1\right)}{\Gamma\left(\alpha+\alpha'-\gamma-\rho+1\right)}x^{\alpha+\alpha'-\gamma-\rho}.
\end{equation}
Also by using reflection formula for the gamma function  \cite{Erdelyi1953}, we obtain
\begin{eqnarray}\label{l4.9}
\Gamma\left(-\alpha-\alpha'+\gamma+\rho-m\right)&\times&\Gamma\left(1-\left(-\alpha-\alpha'+\gamma+\rho-m\right)\right)\nonumber\\
&=&\frac{\pi}{\sin{\left(-\alpha-\alpha'+\gamma+\rho-m\right)\pi}}\nonumber\\
&=&\frac{\pi}{(-1)^m\sin{\left(-\alpha-\alpha'+\gamma+\rho\right)\pi}}.
\end{eqnarray}
Similarly,
\begin{equation}\label{l4.10}
\Gamma\left(-\alpha-\alpha'+\gamma+\rho\right)\Gamma\left(1-\left(-\alpha-\alpha'+\gamma+\rho\right)\right)=\frac{\pi}{\sin{\left(-\alpha-\alpha'+\gamma+\rho\right)\pi}}.
\end{equation}
\end{itemize}
On substituting (\ref{l4.8})-(\ref{l4.10}) in rhs of (\ref{l4.7}), lemma is proved.
\end{proof}
\begin{remark}\label{r4.1}
In view of (\ref{1.15}) and (\ref{1.16}), we have the following correspondence between Lemma \ref{l2.1} and Lemma \ref{ll4.1}:
\begin{itemize}
\item[(a)] If in the hypothesis of Lemma \ref{l2.1}(a), we make the changes $$\alpha\rightarrow-\alpha',\ \alpha'\rightarrow-\alpha,\ \beta\rightarrow-\beta'+[\operatorname{Re}(\gamma)]+1,\ \beta'\rightarrow-\beta,\ \gamma\rightarrow-\gamma+[\operatorname{Re}(\gamma)]+1$$ and in the rhs of assertion $\alpha\rightarrow-\alpha'$, $\alpha'\rightarrow-\alpha$, $\beta\rightarrow-\beta'$, $\beta'\rightarrow-\beta$, $\gamma\rightarrow-\gamma$ respectively, then Lemma \ref{ll4.1}(a) follows. Similarly, 
\item[(b)] If in the hypothesis of Lemma \ref{l2.1}(b), we make the changes $$\alpha\rightarrow-\alpha',\ \alpha'\rightarrow-\alpha,\ \beta\rightarrow-\beta',\ \beta'\rightarrow-\beta+[\operatorname{Re}(\gamma)]+1,\ \gamma\rightarrow-\gamma+[\operatorname{Re}(\gamma)]+1$$ and in the rhs of assertion $\alpha\rightarrow-\alpha'$, $\alpha'\rightarrow-\alpha$, $\beta\rightarrow-\beta'$, $\beta'\rightarrow-\beta$, $\gamma\rightarrow-\gamma$ respectively, then Lemma \ref{ll4.1}(b) is obtained.
\end{itemize}
\end{remark}
\noindent Next theorem give the image of ${}_{p}\Psi_q^k$ under left-hand sided MSM fractional derivative.
\begin{theorem}\label{t4.1}
Let $\alpha,\alpha',\beta,\beta',\gamma,\rho\in\mathbb{C}$ and $k\in\mathbb{R}^+$ be such that $\operatorname{Re}\left(\frac{\rho}{k}\right)>\max\{0,$ $\operatorname{Re}(-\alpha+\beta),\operatorname{Re}(-\alpha-\alpha'-\beta'+\gamma)\}$. Also, let $a\in\mathbb{C}$ and $\mu>0$. If $\Delta>-1$ in (\ref{2.1}), then for $x>0$
\begin{eqnarray}\label{4.1}
&&\left(D_{0+}^{\alpha,\alpha',\beta,\beta',\gamma}\left(t^{\frac{\rho}{k}-1}{}_{p}\Psi_q^k\Bigg[\left.
\begin{matrix}
    (a_i,\alpha_i)_{1,p}\\ 
    (b_j,\beta_j)_{1,q}
  \end{matrix}
\right|at^{\frac{\mu}{k}}\Bigg]\right)\right)(x)\nonumber\\
&=&k^{-\gamma}x^{\alpha+\alpha'-\gamma+\frac{\rho}{k}-1}{}_{p+3}\Psi_{q+3}^k\Bigg[
\begin{matrix}
    (a_i,\alpha_i)_{1,p}&(\rho,\mu)\\ 
    (b_j,\beta_j)_{1,q}&(-k\beta+\rho,\mu)
  \end{matrix}\nonumber\\
& &\left.
\begin{matrix}
   (k\alpha-k\beta+\rho,\mu)&(k\alpha+k\alpha'+k\beta'-k\gamma+\rho,\mu)\\ 
   (k\alpha+k\alpha'-k\gamma+\rho,\mu)&(k\alpha+k\beta'-k\gamma+\rho,\mu)
  \end{matrix}
\right|ax^{\frac{\mu}{k}}\Bigg].
\end{eqnarray}
\end{theorem}
\begin{proof}
After interchanging differentiation and summation, lhs of (\ref{4.1}), using (\ref{1.2}), becomes
\begin{eqnarray*}\label{4.2}
&&\left(D_{0+}^{\alpha,\alpha',\beta,\beta',\gamma}\left(t^{\frac{\rho}{k}-1}\sum_{n=0}^{\infty}\frac{\prod_{i=1}^{p}\Gamma_k(a_i+n\alpha_i)}{\prod_{j=1}^{q}\Gamma_k(b_j+n\beta_j)}\frac{(at^{\frac{\mu}{k}})^n}{n!}\right)\right)(x)\\
&=&\sum_{n=0}^{\infty}\frac{\prod_{i=1}^{p}\Gamma_k(a_i+n\alpha_i)}{\prod_{j=1}^{q}\Gamma_k(b_j+n\beta_j)}\frac{a^n}{n!}\left(D_{0+}^{\alpha,\alpha',\beta,\beta',\gamma}t^{\frac{\rho}{k}+\frac{n\mu}{k}-1}\right)(x)\\
&=&\sum_{n=0}^{\infty}\frac{\prod_{i=1}^{p}\Gamma_k(a_i+n\alpha_i)}{\prod_{j=1}^{q}\Gamma_k(b_j+n\beta_j)}\frac{a^n}{n!}\frac{\Gamma\left(\frac{\rho}{k}+\frac{n\mu}{k}\right)}{\Gamma\left(-\beta+\frac{\rho}{k}+\frac{n\mu}{k}\right)}\\
&&\times\frac{\Gamma\left(\alpha-\beta+\frac{\rho}{k}+\frac{n\mu}{k}\right)\Gamma\left(\alpha+\alpha'+\beta'-\gamma+\frac{\rho}{k}+\frac{n\mu}{k}\right)}{\Gamma\left(\alpha+\alpha'-\gamma+\frac{\rho}{k}+\frac{n\mu}{k}\right)\Gamma\left(\alpha+\beta'-\gamma+\frac{\rho}{k}+\frac{n\mu}{k}\right)}x^{\alpha+\alpha'-\gamma+\frac{\rho}{k}+\frac{n\mu}{k}-1}.\\
&=&k^{-\gamma}x^{\alpha+\alpha'-\gamma+\frac{\rho}{k}-1}\sum_{n=0}^{\infty}\frac{\prod_{i=1}^{p}\Gamma_k(a_i+n\alpha_i)}{\prod_{j=1}^{q}\Gamma_k(b_j+n\beta_j)}\frac{\Gamma_k\left(\rho+n\mu\right)}{\Gamma_k\left(-k\beta+\rho+n\mu\right)}\\
&&\times\frac{\Gamma_k\left(k\alpha-k\beta+\rho+n\mu\right)\Gamma_k\left(k\alpha+k\alpha'+k\beta'-k\gamma+\rho+n\mu\right)}{\Gamma_k\left(k\alpha+k\alpha'-k\gamma+\rho+n\mu\right)\Gamma_k\left(k\alpha+k\beta'-k\gamma+\rho+n\mu\right)}\frac{(ax^{\frac{\mu}{k}})^n}{n!},
\end{eqnarray*}
where we have used (\ref{l4.1}) and (\ref{1.6}). Finally by using (\ref{1.2}), the result follows.
\end{proof}
\noindent Following corollaries follow immediately.
\begin{corollary}\label{c4.1}
Let $\alpha,\beta,\gamma,\rho\in\mathbb{C}$ and $k\in\mathbb{R}^+$ be such that $\operatorname{Re}\left(\frac{\rho}{k}\right)>\max\{0$, $\operatorname{Re}(-\alpha-\beta-\gamma)\}$, and also let $a\in\mathbb{C},\ \mu>0$. If $\Delta>-1$ in (\ref{2.1}), then the left-hand sided generalized fractional differentiation $D_{0+}^{\alpha,\beta,\gamma}$ of ${}_{p}\Psi_q^k$ is given for $x>0$ by
\begin{eqnarray*}\label{4.3}
&&\left(D_{0+}^{\alpha,\beta,\gamma}\left(t^{\frac{\rho}{k}-1}{}_{p}\Psi_q^k\Bigg[\left.
\begin{matrix}
    (a_i,\alpha_i)_{1,p}\\ 
    (b_j,\beta_j)_{1,q}
  \end{matrix}
\right|at^{\frac{\mu}{k}}\Bigg]\right)\right)(x)\nonumber\\
&&=k^{-\alpha}x^{\beta+\frac{\rho}{k}-1}{}_{p+2}\Psi_{q+2}^k\Bigg[\left.
\begin{matrix}
    (a_i,\alpha_i)_{1,p}&(\rho,\mu)&(k\alpha+k\beta+k\gamma+\rho,\mu)\\ 
    (b_j,\beta_j)_{1,q}&(k\beta+\rho,\mu)&(k\gamma+\rho,\mu)
  \end{matrix}
\right|ax^{\frac{\mu}{k}}\Bigg].
\end{eqnarray*}
\end{corollary}
\begin{corollary}
Let $\alpha,\gamma,\rho\in\mathbb{C}$ and $k\in\mathbb{R}^+$ be such that $\operatorname{Re}\left(\frac{\rho}{k}\right)>\max\{0,\operatorname{Re}(-\alpha-\gamma)\}$, and also let $a\in\mathbb{C},\ \mu>0$. If $\Delta>-1$ in (\ref{2.1}), then the left-hand sided Erd\'elyi-Kober fractional differentiation $D_{\gamma,\alpha}^{+}$ $(=D_{0+}^{\alpha,0,\gamma})$ of ${}_{p}\Psi_q^k$ is given for $x>0$ by
\begin{eqnarray*}\label{4.5}
&&\left(D_{\gamma,\alpha}^{+}\left(t^{\frac{\rho}{k}-1}{}_{p}\Psi_q^k\Bigg[\left.
\begin{matrix}
    (a_i,\alpha_i)_{1,p}\\ 
    (b_j,\beta_j)_{1,q}
  \end{matrix}
\right|at^{\frac{\mu}{k}}\Bigg]\right)\right)(x)\nonumber\\
&&=k^{-\alpha}x^{\frac{\rho}{k}-1}{}_{p+1}\Psi_{q+1}^k\Bigg[\left.
\begin{matrix}
    (a_i,\alpha_i)_{1,p}&(k\alpha+k\gamma+\rho,\mu)\\ 
    (b_j,\beta_j)_{1,q}&(k\gamma+\rho,\mu)
  \end{matrix}
\right|ax^{\frac{\mu}{k}}\Bigg].
\end{eqnarray*}
\end{corollary}
\noindent The next theorem yields the right-hand sided MSM fractional derivative of ${}_{p}\Psi_q^k$.
\begin{theorem}\label{t4.2}
Let $\alpha,\alpha',\beta,\beta',\gamma,\rho\in\mathbb{C}$ and $k\in\mathbb{R}^+$ be such that $\operatorname{Re}\left(\frac{\rho}{k}\right)>\max\{\operatorname{Re}(-\beta'),$ $\operatorname{Re}(\alpha'+\beta-\gamma),\operatorname{Re}(\alpha+\alpha'-\gamma)+[\operatorname{Re}(\gamma)]+1,\}$. Also, let $a\in\mathbb{C}$ and $\mu>0$. If $\Delta>-1$ in (\ref{2.1}), then 
\begin{eqnarray}\label{4.6}
&&\left(D_{-}^{\alpha,\alpha',\beta,\beta',\gamma}\left(t^{-\frac{\rho}{k}}{}_{p}\Psi_q^k\Bigg[\left.
\begin{matrix}
    (a_i,\alpha_i)_{1,p}\\ 
    (b_j,\beta_j)_{1,q}
  \end{matrix}
\right|at^{-\frac{\mu}{k}}\Bigg]\right)\right)(x)\nonumber\\
&=&k^{-\gamma}x^{\alpha+\alpha'-\gamma-\frac{\rho}{k}}{}_{p+3}\Psi_{q+3}^k\Bigg[
\begin{matrix}
    (a_i,\alpha_i)_{1,p}&(k\beta'+\rho,\mu)\\ 
    (b_j,\beta_j)_{1,q}&(\rho,\mu)
  \end{matrix}\nonumber\\
& &\left.
\begin{matrix}
   (-k\alpha-k\alpha'+k\gamma+\rho,\mu)&(-k\alpha'-k\beta+k\gamma+\rho,\mu)\\ 
   (-k\alpha'+k\beta'+\rho,\mu)&(-k\alpha-k\alpha'-k\beta+k\gamma+\rho,\mu)
  \end{matrix}
\right|ax^{-\frac{\mu}{k}}\Bigg],
\end{eqnarray}
for $x>0$.
\end{theorem}
\begin{proof}
Using (\ref{l4.2}) and (\ref{1.6}), lhs of (\ref{4.6}) equals
\begin{eqnarray*}
&&\left(D_{-}^{\alpha,\alpha',\beta,\beta',\gamma}\left(t^{-\frac{\rho}{k}}\sum_{n=0}^{\infty}\frac{\prod_{i=1}^{p}\Gamma_k(a_i+n\alpha_i)}{\prod_{j=1}^{q}\Gamma_k(b_j+n\beta_j)}\frac{(at^{-\frac{\mu}{k}})^n}{n!}\right)\right)(x)\\
&=&\sum_{n=0}^{\infty}\frac{\prod_{i=1}^{p}\Gamma_k(a_i+n\alpha_i)}{\prod_{j=1}^{q}\Gamma_k(b_j+n\beta_j)}\frac{a^n}{n!}\left(D_{-}^{\alpha,\alpha',\beta,\beta',\gamma}t^{-\left(\frac{\rho}{k}+\frac{n\mu}{k}\right)}\right)(x)\\
&=&\sum_{n=0}^{\infty}\frac{\prod_{i=1}^{p}\Gamma_k(a_i+n\alpha_i)}{\prod_{j=1}^{q}\Gamma_k(b_j+n\beta_j)}\frac{a^n}{n!}\frac{\Gamma\left(\beta'+\frac{\rho}{k}+\frac{n\mu}{k}\right)}{\Gamma\left(\frac{\rho}{k}+\frac{n\mu}{k}\right)}\nonumber\\
&&\times\frac{\Gamma\left(-\alpha-\alpha'+\gamma+\frac{\rho}{k}+\frac{n\mu}{k}\right)\Gamma\left(-\alpha'-\beta+\gamma+\frac{\rho}{k}+\frac{n\mu}{k}\right)}{\Gamma\left(-\alpha'+\beta'+\frac{\rho}{k}+\frac{n\mu}{k}\right)\Gamma\left(-\alpha-\alpha'-\beta+\gamma+\frac{\rho}{k}+\frac{n\mu}{k}\right)}x^{\alpha+\alpha'-\gamma-\frac{\rho}{k}-\frac{n\mu}{k}}\\
&=&k^{-\gamma}x^{\alpha+\alpha'-\gamma-\frac{\rho}{k}}\sum_{n=0}^{\infty}\frac{\prod_{i=1}^{p}\Gamma_k(a_i+n\alpha_i)}{\prod_{j=1}^{q}\Gamma_k(b_j+n\beta_j)}\frac{\Gamma_k\left(k\beta'+\rho+n\mu\right)}{\Gamma_k\left(\rho+n\mu\right)}\nonumber\\
&&\times\frac{\Gamma_k\left(-k\alpha-k\alpha'+k\gamma+\rho+n\mu\right)\Gamma_k\left(-k\alpha'-k\beta+k\gamma+\rho+n\mu\right)}{\Gamma_k\left(-k\alpha'+k\beta'+\rho+n\mu\right)\Gamma_k\left(-k\alpha-k\alpha'-k\beta+k\gamma+\rho+n\mu\right)}\frac{(ax^{-\frac{\mu}{k}})^n}{n!},
\end{eqnarray*}
and thus theorem is proved using (\ref{1.2}).
\end{proof}
\begin{corollary}\label{c4.4}
Let $\alpha,\beta,\gamma,\rho\in\mathbb{C}$ and $k\in\mathbb{R}^+$ be such that $\operatorname{Re}\left(\frac{\rho}{k}\right)>\max\{\operatorname{Re}(-\alpha-\gamma),$ $\operatorname{Re}(\beta)+[\operatorname{Re}(\alpha)]+1\}$, and also let $a\in\mathbb{C},\ \mu>0$. If $\Delta>-1$ in (\ref{2.1}), then the right-hand sided generalized fractional differentiation $D_{-}^{\alpha,\beta,\gamma}$ of ${}_{p}\Psi_q^k$ is given for $x>0$ by
\begin{eqnarray*}\label{4.11}
&&\left(D_{-}^{\alpha,\beta,\gamma}\left(t^{-\frac{\rho}{k}}{}_{p}\Psi_q^k\Bigg[\left.
\begin{matrix}
    (a_i,\alpha_i)_{1,p}\\ 
    (b_j,\beta_j)_{1,q}
  \end{matrix}
\right|at^{-\frac{\mu}{k}}\Bigg]\right)\right)(x)\nonumber\\
&&=k^{-\alpha}x^{\beta-\frac{\rho}{k}}{}_{p+2}\Psi_{q+2}^k\Bigg[\left.
\begin{matrix}
    (a_i,\alpha_i)_{1,p}&(-k\beta+\rho,\mu)&(k\alpha+k\gamma+\rho,\mu)\\ 
    (b_j,\beta_j)_{1,q}&(\rho,\mu)&(-k\beta+k\gamma+\rho,\mu)
  \end{matrix}
\right|ax^{-\frac{\mu}{k}}\Bigg].
\end{eqnarray*}
\end{corollary}
\begin{corollary}\label{c4.6}
Let $\alpha,\gamma,\rho\in\mathbb{C}$ and $k\in\mathbb{R}^+$ be such that $\operatorname{Re}\left(\frac{\rho}{k}\right)>\max\{[\operatorname{Re}(\alpha)]+1,$ $\operatorname{Re}(-\alpha-\gamma)\}$, and also let $a\in\mathbb{C},\ \mu>0$. If $\Delta>-1$ in (\ref{2.1}), then the right-hand sided Erd\'elyi-Kober fractional differentiation $D_{\gamma,\alpha}^{-}$ $(=D_{-}^{\alpha,0,\gamma})$ of ${}_{p}\Psi_q^k$ is given for $x>0$ by
\begin{eqnarray*}\label{4.13}
&&\left(D_{\gamma,\alpha}^{-}\left(t^{-\frac{\rho}{k}}{}_{p}\Psi_q^k\Bigg[\left.
\begin{matrix}
    (a_i,\alpha_i)_{1,p}\\ 
    (b_j,\beta_j)_{1,q}
  \end{matrix}
\right|at^{-\frac{\mu}{k}}\Bigg]\right)\right)(x)\nonumber\\
&&=k^{-\alpha}x^{-\frac{\rho}{k}}{}_{p+1}\Psi_{q+1}^k\Bigg[\left.
\begin{matrix}
    (a_i,\alpha_i)_{1,p}&(k\alpha+k\gamma+\rho,\mu)\\ 
    (b_j,\beta_j)_{1,q}&(k\gamma+\rho,\mu)
  \end{matrix}
\right|ax^{-\frac{\mu}{k}}\Bigg].
\end{eqnarray*}
\end{corollary}
\begin{remark}
In view of Remark \ref{r4.1}, Theorems \ref{t4.1} and \ref{t4.2} respectively follow from Theorems \ref{t3.1} and \ref{t3.2}.  
\end{remark}

\section{Caputo-type MSM fractional differentiation of ${}_{p}\Psi_q^k$}
\setcounter{equation}{0}
The Riemann-Liouville derivatives lead to difficulties while applying to real world problems, especially in the context of initial conditions. Hence, Caputo derivative is being extensively used in applications as the initial conditions are physically meaningful. Rao \textit{et al.} \cite{Rao15}, introduced Caputo-type fractional derivative, which involve Gauss hypergeometric function in the kernel. For $\alpha,\beta,\gamma\in\mathbb{C}$ and $x\in\mathbb{R}^+$ with $\operatorname{Re}(\alpha)>0$, the left- and right-hand sided Caputo fractional differential operators associated with Gauss hypergeometric function are defined by
\begin{equation}\label{c5.1}
\left({}^{c}D_{0+}^{\alpha,\beta,\gamma}f\right)(x)=\left(I_{0+}^{-\alpha+[\operatorname{Re}(\alpha)]+1,-\beta-[\operatorname{Re}(\alpha)]-1,\alpha+\gamma-[\operatorname{Re}(\alpha)]-1}f^{([\operatorname{Re}(\alpha)]+1)}\right)(x)
\end{equation}
and
\begin{equation}\label{c5.2}
\left({}^{c}D_{-}^{\alpha,\beta,\gamma}f\right)(x)=(-1)^{[\operatorname{Re}(\alpha)]+1}\left(I_{-}^{-\alpha+[\operatorname{Re}(\alpha)]+1,-\beta-[\operatorname{Re}(\alpha)]-1,\alpha+\gamma}f^{\left([\operatorname{Re}(\alpha)]+1\right)}\right)(x).
\end{equation}
The relation between the Caputo-type MSM fractional derivative and the MSM fractional derivative is same as the relation between the Caputo fractional derivative and the Riemann-Liouville fractional derivative, \textit{i.e.}, the operations of integration and differentiation
are interchanged in the corresponding definitions. For $\alpha,\alpha',\beta,\beta',\gamma\in\mathbb{C}$ and $x\in\mathbb{R}^+$ with $\operatorname{Re}(\gamma)>0$, the left- and right-hand sided Caputo-type MSM fractional differential operators associated with third Appell function are defined by
\begin{equation}\label{5.1}
\left({}^{c}D_{0+}^{\alpha,\alpha',\beta,\beta',\gamma}f\right)(x)=\left(I_{0+}^{-\alpha',-\alpha,-\beta'+[\operatorname{Re}(\gamma)]+1,-\beta,-\gamma+[\operatorname{Re}(\gamma)]+1}f^{([\operatorname{Re}(\gamma)]+1)}\right)(x)
\end{equation}
and
\begin{equation}\label{5.2}
\left({}^{c}D_{-}^{\alpha,\alpha',\beta,\beta',\gamma}f\right)(x)=\left(-1\right)^{[\operatorname{Re}(\gamma)]+1}\left(I_{-}^{-\alpha',-\alpha,-\beta',-\beta+[\operatorname{Re}(\gamma)]+1,-\gamma+[\operatorname{Re}(\gamma)]+1}f^{({[\operatorname{Re}(\gamma)]+1})}\right)(x),
\end{equation}
respectively, where $f^{(n)}$ denotes the $n$-th derivative of $f$. Fractional operators (\ref{5.1}) and (\ref{5.2}) are connected to (\ref{c5.1}) and (\ref{c5.2}) as follows:
\begin{equation}\label{c1.20}
\left({}^{c}D_{0+}^{0,\alpha',\beta,\beta',\gamma}f\right)(x)=\left({}^{c}D_{0+}^{\gamma,\alpha'-\gamma,\beta'-\gamma}f\right)(x),\ \ \ \ 
\left({}^{c}D_{-}^{0,\alpha',\beta,\beta',\gamma}f\right)(x)=\left({}^{c}D_{-}^{\gamma,\alpha'-\gamma,\beta'-\gamma}f\right)(x).
\end{equation}
In this section, we study the Caputo-type Marichev-Saigo-Maeda fractional differentiation of the generalized $K$-Wright function.

\begin{lemma}\label{ll5.1}
Let $\alpha,\alpha',\beta,\beta',\gamma,\rho\in\mathbb{C}$ and $m=[\operatorname{Re}(\gamma)]+1$.
\begin{itemize}
\item[(a)] If $\operatorname{Re}(\rho)-m>$ $\max\{0,\operatorname{Re}(-\alpha+\beta),\operatorname{Re}(-\alpha-\alpha'-\beta'+\gamma)\}$, then
\begin{equation}\label{l5.1}
\left({}^{c}D_{0+}^{\alpha,\alpha',\beta,\beta',\gamma}t^{\rho-1}\right)(x)=\frac{\Gamma(\rho)\Gamma(\alpha-\beta+\rho-m)\Gamma(\alpha+\alpha'+\beta'-\gamma+\rho-m)}{\Gamma(-\beta+\rho-m)\Gamma(\alpha+\alpha'-\gamma+\rho)\Gamma(\alpha+\beta'-\gamma+\rho-m)}x^{\alpha+\alpha'-\gamma+\rho-1}.
\end{equation}
\item[(b)] If $\operatorname{Re}(\rho)+m>$ $\max\{\operatorname{Re}(-\beta'),\operatorname{Re}(\alpha'+\beta-\gamma),\operatorname{Re}(\alpha+\alpha'-\gamma)+[\operatorname{Re}(\gamma)]+1\}$, then
\begin{equation}\label{l5.2}
\left({}^{c}D_{-}^{\alpha,\alpha',\beta,\beta',\gamma}t^{-\rho}\right)(x)=\frac{\Gamma\left(\beta'+\rho+m\right)\Gamma\left(-\alpha-\alpha'+\gamma+\rho\right)\Gamma\left(-\alpha'-\beta+\gamma+\rho+m\right)}{\Gamma\left(\rho\right)\Gamma\left(-\alpha'+\beta'+\rho+m\right)\Gamma\left(-\alpha-\alpha'-\beta+\gamma+\rho+m\right)}x^{\alpha+\alpha'-\gamma-\rho}.
\end{equation}
\end{itemize} 
\end{lemma}
\begin{proof} 
\begin{itemize}
\item[(a)] From (\ref{5.1}), we have
\begin{eqnarray*}
\left({}^{c}D_{0+}^{\alpha,\alpha',\beta,\beta',\gamma}t^{\rho-1}\right)(x)
&=&\left(I_{0+}^{-\alpha',-\alpha,-\beta'+m,-\beta,-\gamma+m}\frac{d^m}{\,dt^m}t^{\rho-1}\right)(x)\\
&=&\frac{\Gamma(\rho)}{\Gamma(\rho-m)}\left(I_{0+}^{-\alpha',-\alpha,-\beta'+m,-\beta,-\gamma+m}t^{\rho-m-1}\right)(x),
\end{eqnarray*}
which on using (\ref{2.4}) gives (\ref{l5.1}).
\item[(b)] From (\ref{5.2}), we have
\begin{eqnarray*}\label{l5.7}
\left({}^{c}D_{-}^{\alpha,\alpha',\beta,\beta',\gamma}t^{-\rho}\right)(x)
&=&\left(-1\right)^m\left(I_{-}^{-\alpha',-\alpha,-\beta',-\beta+m,-\gamma+m}\frac{d^m}{\,dt^m}t^{-\rho}\right)(x)\\
&=&\frac{\Gamma(\rho+m)}{\Gamma(\rho)}\left(I_{-}^{-\alpha',-\alpha,-\beta',-\beta+m,-\gamma+m}t^{-\rho-m}\right)(x),
\end{eqnarray*}
and thus (\ref{l5.2}) follows from (\ref{2.5}).
\end{itemize}
\end{proof}
\begin{theorem}\label{t5.1}
Let $\alpha,\alpha',\beta,\beta',\gamma,\rho\in\mathbb{C}$, $m=[\operatorname{Re}(\gamma)]+1$ and $k\in\mathbb{R}^+$ be such that $\operatorname{Re}\left(\frac{\rho}{k}\right)-m>\max\{0,\operatorname{Re}(-\alpha+\beta),\operatorname{Re}(-\alpha-\alpha'-\beta'+\gamma)\}$. Also, let $a\in\mathbb{C}$ and $\mu>0$. If $\Delta>-1$ in (\ref{2.1}), then for $x>0$
\begin{eqnarray}\label{5.3}
&&\left({}^{c}D_{0+}^{\alpha,\alpha',\beta,\beta',\gamma}\left(t^{\frac{\rho}{k}-1}{}_{p}\Psi_q^k\Bigg[\left.
\begin{matrix}
    (a_i,\alpha_i)_{1,p}\\ 
    (b_j,\beta_j)_{1,q}
  \end{matrix}
\right|at^{\frac{\mu}{k}}\Bigg]\right)\right)(x)\nonumber\\
&=&k^{-\gamma}x^{\alpha+\alpha'-\gamma+\frac{\rho}{k}-1}{}_{p+3}\Psi_{q+3}^k\Bigg[
\begin{matrix}
    (a_i,\alpha_i)_{1,p}&(\rho,\mu)\\ 
    (b_j,\beta_j)_{1,q}&(-k\beta+\rho-km,\mu)
  \end{matrix}\nonumber\\
& &\left.
\begin{matrix}
    (k\alpha-k\beta+\rho-km,\mu)&(k\alpha+k\alpha'+k\beta'-k\gamma+\rho-km,\mu)\\ 
    (k\alpha+k\alpha'-k\gamma+\rho,\mu)&(k\alpha+k\beta'-k\gamma+\rho-km,\mu)
  \end{matrix}
\right|ax^{\frac{\mu}{k}}\Bigg].
\end{eqnarray}
\end{theorem}
\begin{proof}
By using (\ref{1.6}) and (\ref{l5.1}), we get
\begin{eqnarray*}\label{5.4}
&&\left({}^{c}D_{0+}^{\alpha,\alpha',\beta,\beta',\gamma}\left(t^{\frac{\rho}{k}-1}{}_{p}\Psi_q^k\Bigg[\left.
\begin{matrix}
    (a_i,\alpha_i)_{1,p}\\ 
    (b_j,\beta_j)_{1,q}
  \end{matrix}
\right|at^{\frac{\mu}{k}}\Bigg]\right)\right)(x)\nonumber\\
&=&\sum_{n=0}^{\infty}\frac{\prod_{i=1}^{p}\Gamma_k(a_i+n\alpha_i)}{\prod_{j=1}^{q}\Gamma_k(b_j+n\beta_j)}\frac{a^n}{n!}\left({}^{c}D_{0+}^{\alpha,\alpha',\beta,\beta',\gamma}t^{\frac{\rho}{k}+\frac{n\mu}{k}-1}\right)(x)\\
&=&\sum_{n=0}^{\infty}\frac{\prod_{i=1}^{p}\Gamma_k(a_i+n\alpha_i)}{\prod_{j=1}^{q}\Gamma_k(b_j+n\beta_j)}\frac{a^n}{n!} \frac{\Gamma(\frac{\rho}{k}+\frac{n\mu}{k})}{\Gamma(-\beta+\frac{\rho}{k}+\frac{n\mu}{k}-m)}\\
&&\times\frac{\Gamma\left(\alpha-\beta+\frac{\rho}{k}+\frac{n\mu}{k}-m\right)\Gamma\left(\alpha+\alpha'+\beta'-\gamma+\frac{\rho}{k}+\frac{n\mu}{k}-m\right)}{\Gamma\left(\alpha+\alpha'-\gamma+\frac{\rho}{k}+\frac{n\mu}{k}\right)\Gamma\left(\alpha+\beta'-\gamma+\frac{\rho}{k}+\frac{n\mu}{k}-m\right)}x^{\alpha+\alpha'-\gamma+\frac{\rho}{k}+\frac{n\mu}{k}-1}\\
&=&k^{-\gamma}x^{\alpha+\alpha'-\gamma+\frac{\rho}{k}-1}\sum_{n=0}^{\infty}\frac{\prod_{i=1}^{p}\Gamma_k(a_i+n\alpha_i)}{\prod_{j=1}^{q}\Gamma_k(b_j+n\beta_j)}\frac{\Gamma_k\left(\rho+n\mu\right)}{\Gamma_k\left(-k\beta+\rho+n\mu-km\right)}\\
&&\times\frac{\Gamma_k\left(k\alpha-k\beta+\rho+n\mu-km\right)\Gamma_k\left(k\alpha+k\alpha'+k\beta'-k\gamma+\rho+n\mu-km\right)}{\Gamma_k\left(k\alpha+k\alpha'-k\gamma+\rho+n\mu\right)\Gamma_k\left(k\alpha+k\beta'-k\gamma+\rho+n\mu-km\right)}\frac{(ax^{\frac{\mu}{k}})^n}{n!},
\end{eqnarray*}
and thus the proof is complete by using (\ref{1.2}).
\end{proof}
\begin{corollary}\label{c5.1}
Let $\alpha,\beta,\gamma,\rho\in\mathbb{C}$, $m=[\operatorname{Re}(\alpha)]+1$ and $k\in\mathbb{R}^+$ be such that $\operatorname{Re}\left(\frac{\rho}{k}\right)-m>\max\{0,\operatorname{Re}(-\alpha-\beta-\gamma)\}$, and also let $a\in\mathbb{C},\ \mu>0$. If $\Delta>-1$ in (\ref{2.1}), then the left-hand sided generalized Caputo fractional differentiation ${}^{c}D_{0+}^{\alpha,\beta,\gamma}$ of ${}_{p}\Psi_q^k$ is given for $x>0$ by
\begin{eqnarray*}\label{5.6}
&&\left({}^{c}D_{0+}^{\alpha,\beta,\gamma}\left(t^{\frac{\rho}{k}-1}{}_{p}\Psi_q^k\Bigg[\left.
\begin{matrix}
    (a_i,\alpha_i)_{1,p}\\ 
    (b_j,\beta_j)_{1,q}
  \end{matrix}
\right|at^{\frac{\mu}{k}}\Bigg]\right)\right)(x)\nonumber\\
&&=k^{-\alpha}x^{\beta+\frac{\rho}{k}-1}{}_{p+2}\Psi_{q+2}^k\Bigg[\left.
\begin{matrix}
    (a_i,\alpha_i)_{1,p}&(\rho,\mu)&(k\alpha+k\beta+k\gamma+\rho-km,\mu)\\ 
    (b_j,\beta_j)_{1,q}&(k\beta+\rho,\mu)&(k\gamma+\rho-km,\mu)
  \end{matrix}
\right|ax^{\frac{\mu}{k}}\Bigg].
\end{eqnarray*}
\end{corollary}
\begin{corollary}\label{c5.3}
Let $\alpha,\gamma,\rho\in\mathbb{C}$, $m=[\operatorname{Re}(\alpha)]+1$ and $k\in\mathbb{R}^+$ be such that $\operatorname{Re}\left(\frac{\rho}{k}\right)-m>\max\{0,\operatorname{Re}(-\alpha-\gamma)\}$, and also let $a\in\mathbb{C},\ \mu>0$. If $\Delta>-1$ in (\ref{2.1}), then the left-hand sided Caputo-type Erd¶elyi-Kober fractional differentiation ${}^{c}D_{\gamma,\alpha}^{+}$ $(={}^{c}D_{0+}^{\alpha,0,\gamma})$ of ${}_{p}\Psi_q^k$ is given for $x>0$ by
\begin{eqnarray*}\label{5.8}
&&\left({}^{c}D_{\gamma,\alpha}^{+}\left(t^{\frac{\rho}{k}-1}{}_{p}\Psi_q^k\Bigg[\left.
\begin{matrix}
    (a_i,\alpha_i)_{1,p}\\ 
    (b_j,\beta_j)_{1,q}
  \end{matrix}
\right|at^{\frac{\mu}{k}}\Bigg]\right)\right)(x)\nonumber\\
&&=k^{-\alpha}x^{\frac{\rho}{k}-1}{}_{p+1}\Psi_{q+1}^k\Bigg[\left.
\begin{matrix}
    (a_i,\alpha_i)_{1,p}&(k\alpha+k\gamma+\rho-km,\mu)\\ 
    (b_j,\beta_j)_{1,q}&(k\gamma+\rho-km,\mu)
  \end{matrix}
\right|ax^{\frac{\mu}{k}}\Bigg].
\end{eqnarray*}
\end{corollary}
\begin{theorem}\label{t5.2}
Let $\alpha,\alpha',\beta,\beta',\gamma,\rho\in\mathbb{C}$, $m=[\operatorname{Re}(\gamma)]+1$ and $k\in\mathbb{R}^+$ be such that $\operatorname{Re}\left(\frac{\rho}{k}\right)+m>\max\{\operatorname{Re}(-\beta'),\operatorname{Re}(\alpha'+\beta-\gamma),\operatorname{Re}(\alpha+\alpha'-\gamma)+m\}$. Also, let $a\in\mathbb{C}$ and $\mu>0$. If $\Delta>-1$ in (\ref{2.1}), then 
\begin{eqnarray}\label{5.9}
&&\left({}^{c}D_{-}^{\alpha,\alpha',\beta,\beta',\gamma}\left(t^{-\frac{\rho}{k}}{}_{p}\Psi_q^k\Bigg[\left.
\begin{matrix}
    (a_i,\alpha_i)_{1,p}\\ 
    (b_j,\beta_j)_{1,q}
  \end{matrix}
\right|at^{-\frac{\mu}{k}}\Bigg]\right)\right)(x)\nonumber\\
&=&k^{-\gamma}x^{\alpha+\alpha'-\gamma-\frac{\rho}{k}}{}_{p+3}\Psi_{q+3}^k\Bigg[
\begin{matrix}
    (a_i,\alpha_i)_{1,p}&(k\beta'+\rho+km,\mu)\\ 
    (b_j,\beta_j)_{1,q}&(\rho,\mu)
  \end{matrix}\nonumber\\
& &\left.
\begin{matrix}
    (-k\alpha-k\alpha'+k\gamma+\rho,\mu)&(-k\alpha'-k\beta+k\gamma+\rho+km,\mu)\\ 
    (-k\alpha'+k\beta'+\rho+km,\mu)&(-k\alpha-k\alpha'-k\beta+k\gamma+\rho+km,\mu)
  \end{matrix}
\right|ax^{-\frac{\mu}{k}}\Bigg],
\end{eqnarray}
for $x>0$.
\end{theorem}
\begin{proof}
By using (\ref{1.6}) and (\ref{l5.2}), we have
\begin{eqnarray*}
&&\left({}^{c}D_{-}^{\alpha,\alpha',\beta,\beta',\gamma}\left(t^{-\frac{\rho}{k}}{}_{p}\Psi_q^k\Bigg[\left.
\begin{matrix}\label{5.7}
    (a_i,\alpha_i)_{1,p}\\ 
    (b_j,\beta_j)_{1,q}
  \end{matrix}
\right|at^{-\frac{\mu}{k}}\Bigg]\right)\right)(x)\\
&=&\sum_{n=0}^{\infty}\frac{\prod_{i=1}^{p}\Gamma_k(a_i+n\alpha_i)}{\prod_{j=1}^{q}\Gamma_k(b_j+n\beta_j)}\frac{a^n}{n!}\left({}^{c}D_{-}^{\alpha,\alpha',\beta,\beta',\gamma}t^{-\frac{\rho}{k}-\frac{n\mu}{k}}\right)(x)\\
&=&\sum_{n=0}^{\infty}\frac{\prod_{i=1}^{p}\Gamma_k(a_i+n\alpha_i)}{\prod_{j=1}^{q}\Gamma_k(b_j+n\beta_j)}\frac{a^n}{n!}\frac{\Gamma\left(\beta'+\frac{\rho}{k}+\frac{n\mu}{k}+m\right)}{\Gamma\left(\frac{\rho}{k}+\frac{n\mu}{k}\right)}\\
& &\times\frac{\Gamma\left(-\alpha-\alpha'+\gamma+\frac{\rho}{k}+\frac{n\mu}{k}\right)\Gamma\left(-\alpha'-\beta+\gamma+\frac{\rho}{k}+\frac{n\mu}{k}+m\right)}{\Gamma\left(-\alpha'+\beta'+\frac{\rho}{k}+\frac{n\mu}{k}+m\right)\Gamma\left(-\alpha-\alpha'-\beta+\gamma+\frac{\rho}{k}+\frac{n\mu}{k}+m\right)}x^{\alpha+\alpha'-\gamma-\frac{\rho}{k}-\frac{n\mu}{k}}\\
&=&k^{-\gamma}x^{\alpha+\alpha'-\gamma-\frac{\rho}{k}}\sum_{n=0}^{\infty}\frac{\prod_{i=1}^{p}\Gamma_k(a_i+n\alpha_i)}{\prod_{j=1}^{q}\Gamma_k(b_j+n\beta_j)}\frac{\Gamma_k\left(k\beta'+\rho+n\mu+km\right)}{\Gamma_k\left(\rho+n\mu\right)}\\
& &\times\frac{\Gamma_k\left(-k\alpha-k\alpha'+k\gamma+\rho+n\mu\right)\Gamma_k\left(-k\alpha'-k\beta+k\gamma+\rho+n\mu+km\right)}{\Gamma_k\left(-k\alpha'+k\beta'+\rho+n\mu+km\right)\Gamma_k\left(-k\alpha-k\alpha'-k\beta+k\gamma+\rho+n\mu+km\right)}\frac{(ax^{-\frac{\mu}{k}})^n}{n!},
\end{eqnarray*}
and thus the theorem follows from (\ref{1.2}).
\end{proof}
\begin{corollary}\label{c5.4}
Let $\alpha,\beta,\gamma,\rho\in\mathbb{C}$, $m=[\operatorname{Re}(\alpha)]+1$ and $k\in\mathbb{R}^+$ be such that $\operatorname{Re}\left(\frac{\rho}{k}\right)+m>\max\{\operatorname{Re}(\beta)+m,\operatorname{Re}(-\alpha-\gamma)\}$, and also let $a\in\mathbb{C},\ \mu>0$. If $\Delta>-1$ in (\ref{2.1}), then the right-hand sided generalized Caputo fractional differentiation ${}^{c}D_{-}^{\alpha,\beta,\gamma}$ of ${}_{p}\Psi_q^k$ is given for $x>0$ by
\begin{eqnarray*}\label{5.10}
&&\left({}^{c}D_{-}^{\alpha,\beta,\gamma}\left(t^{-\frac{\rho}{k}}{}_{p}\Psi_q^k\Bigg[\left.
\begin{matrix}
    (a_i,\alpha_i)_{1,p}\\ 
    (b_j,\beta_j)_{1,q}
  \end{matrix}
\right|at^{-\frac{\mu}{k}}\Bigg]\right)\right)(x)\nonumber\\
&&=k^{-\alpha}x^{\beta-\frac{\rho}{k}}{}_{p+2}\Psi_{q+2}^k\Bigg[\left.
\begin{matrix}
    (a_i,\alpha_i)_{1,p}&(-k\beta+\rho,\mu)&(k\alpha+k\gamma+\rho+km,\mu)\\ 
    (b_j,\beta_j)_{1,q}&(\rho,\mu)&(-k\beta+k\gamma+\rho+km,\mu)
  \end{matrix}
\right|ax^{-\frac{\mu}{k}}\Bigg].
\end{eqnarray*}
\end{corollary}
\begin{corollary}\label{c5.6}
Let $\alpha,\gamma,\rho\in\mathbb{C}$, $m=[\operatorname{Re}(\alpha)]+1$ and $k\in\mathbb{R}^+$ be such that $\operatorname{Re}\left(\frac{\rho}{k}\right)+m>\max\{m,\operatorname{Re}(-\alpha-\gamma)\}$, and also let $a\in\mathbb{C},\ \mu>0$. If $\Delta>-1$ in (\ref{2.1}), then the right-hand sided Caputo-type Erd¶elyi-Kober fractional differentiation ${}^{c}D_{\gamma,\alpha}^{-}$ $(={}^{c}D_{-}^{\alpha,0,\gamma})$ of ${}_{p}\Psi_q^k$ is given for $x>0$ by
\begin{eqnarray*}\label{5.12}
&&\left({}^{c}D_{\gamma,\alpha}^{-}\left(t^{-\frac{\rho}{k}}{}_{p}\Psi_q^k\Bigg[\left.
\begin{matrix}
    (a_i,\alpha_i)_{1,p}\\ 
    (b_j,\beta_j)_{1,q}
  \end{matrix}
\right|at^{-\frac{\mu}{k}}\Bigg]\right)\right)(x)\nonumber\\
&&=k^{-\alpha}x^{-\frac{\rho}{k}}{}_{p+1}\Psi_{q+1}^k\Bigg[\left.
\begin{matrix}
    (a_i,\alpha_i)_{1,p}&(k\alpha+k\gamma+\rho+km,\mu)\\ 
    (b_j,\beta_j)_{1,q}&(k\gamma+\rho+km,\mu)
  \end{matrix}
\right|ax^{-\frac{\mu}{k}}\Bigg].
\end{eqnarray*}
\end{corollary}
\section{Conclusion}
The Marichev-Saigo-Maeda fractional operators transform the generalized $K$-Wright function into a higher ordered generalized $K$-Wright function. The results obtained by Kilbas \cite{Kilbas2004113} for generalized Wright function, and by  Gehlot and Prajapati \cite{Gehlot283} for generalized $K$-Wright function are particular cases of the results derived in this paper. In view of (\ref{1.19}), (\ref{1.20}) and the fact that Riemann-Liouville, Weyl and Erd\'elyi-Kober fractional calculus operators are special cases of Saigo's operators, the effect of all these fractional operators on the generalized $K$-Wright function and its special cases can easily be obtained from our results.
\printbibliography
\end{document}